\documentclass[10pt]{article}
\usepackage{amsmath,amssymb,amsthm,amsfonts}
\usepackage{enumerate,times,epsfig,color,mathdots,float,subcaption,array}
\usepackage[utf8]{inputenc}
\usepackage{tikz}

\usepackage[letterpaper,twoside,outer=1.3in,vmargin=1.3in,]{geometry}

\newtheorem{THM}{Theorem}

\newtheorem{LE}[THM]{Lemma}
\newtheorem{PR}[THM]{Proposition}
\newtheorem{CN}[THM]{Conjecture}

\newtheorem{RE}[THM]{Remark}

\newtheorem{DE}[THM]{Definition}
\newcommand{\tr}{\triangle}

\newcommand{\ignore}[1]{}

\newcommand{\conv}{\mathrm{conv}}
\newcommand{\supp}{\mathrm{support}}

\newcommand{\si}[1]{\mathrm{si}\!\left(#1\right)}
\newcommand{\cuboid}[1]{\mathrm{cuboid}\!\left(#1\right)}
\newcommand{\cycle}[1]{\mathrm{cycle}\!\left(#1\right)}
\newcommand{\cocycle}[1]{\mathrm{cocycle}\!\left(#1\right)}

\newcommand{\core}[1]{\mathrm{core}\!\left(#1\right)}

\newcounter{claim_nb}[THM]
\newtheorem{CLM}[claim_nb]{Claim}
\newenvironment{cproof}
{\begin{proof}
[Proof of Claim]
\vspace{-1.2\parsep}}
{ 
\end{proof}}

\title{Idealness of $k$-wise intersecting families\footnote{An extended abstract of this work was published under the same title in the proceedings of the $21\textsuperscript{st}$ Conference in Integer Programming and Combinatorial Optimization~\cite{Abdi-2020}.}}

\author{Ahmad Abdi
\thanks{
Department of Mathematics, London School of Economics and Political Science, London, UK
}
 \and G\'{e}rard Cornu\'{e}jols
 \thanks{Tepper School of Business, Carnegie Mellon University, Pittsburgh, USA}
  \and Tony Huynh
  \thanks{School of Mathematics, Monash University, Melbourne, Australia}
   \and Dabeen Lee
   \thanks{Discrete Mathematics Group, Institute for Basic Science (IBS), Daejeon, South Korea}
}

\begin{document}


\maketitle

\begin{abstract}
A clutter is \emph{$k$-wise intersecting} if every $k$ members have a common element, yet no element belongs to all members.  We conjecture that, for some integer $k\geq 4$, every $k$-wise intersecting clutter is non-ideal. As evidence for our conjecture, we prove it for $k=4$ for the class of binary clutters. Two key ingredients for our proof are Jaeger's $8$-flow theorem for graphs, and Seymour's characterization of the binary matroids with the sums of circuits property. As further evidence for our conjecture, we also note that it follows from an unpublished conjecture of Seymour from 1975. We also discuss connections to the chromatic number of a clutter, projective geometries over the two-element field, uniform cycle covers in graphs, and quarter-integral packings of value two in ideal clutters.
\end{abstract}

\section{Introduction}

Let $V$ be a finite set of {\it elements}, and let $\mathcal{C}$ be a family of subsets of $V$ called {\it members}. The family $\mathcal{C}$ is a {\it clutter} over {\it ground set $V$}, if no member contains another one~\cite{Edmonds70}. A {\it cover of $\mathcal{C}$} is a subset $B\subseteq V$ such that $B\cap C\neq \emptyset$ for all $C\in \mathcal{C}$. Consider for $w\in \mathbb{Z}^V_+$ the dual pair of linear programs
\begin{align*}
(P)\quad&\begin{array}{ll} \min \quad  &w^\top x\\ \text{s.t.} & \sum \left(x_u:u\in C\right)\geq 1 \quad \forall C\in \mathcal{C} \\ & x\geq {\bf 0} \end{array}\qquad
(D)\quad&\begin{array}{ll} \max \quad  &{\bf 1}^\top y\\ \text{s.t.} & \sum \left(y_C:u\in C\in \mathcal{C}\right)\leq w_u \quad \forall u\in V\\ & y\geq {\bf 0}.\end{array}
\end{align*} If the dual $(D)$ has an integral optimal solution for every right-hand-side vector $w\in \mathbb{Z}^V_+$, then $\mathcal{C}$ is said to have the {\it max-flow min-cut (MFMC) property}~\cite{Conforti93}. By the theory of {\it totally dual integral} linear systems, for every MFMC clutter, the primal $(P)$ also admits an integral optimal solution for every cost vector $w\in \mathbb{Z}^V_+$~\cite{Edmonds77}. This is why the class of MFMC clutters is a natural host to many beautiful {\it min-max theorems} in Combinatorial Optimization~\cite{Cornuejols01}. Let us elaborate.

The {\it packing number of $\mathcal{C}$}, denoted $\nu(\mathcal{C})$, is the maximum number of pairwise disjoint members. Note that $\nu(\mathcal{C})$ is equal to the maximum value of an integral feasible solution to $(D)$ for $w={\bf 1}$. Furthermore, the covers correspond precisely to the $0-1$ feasible solutions to $(P)$. The {\it covering number of $\mathcal{C}$}, denoted $\tau(\mathcal{C})$, is the minimum cardinality of a cover. Notice that $\tau(\mathcal{C})$ is equal to the minimum value of an integral feasible solution to $(P)$ for $w={\bf 1}$. Also, by Weak LP Duality, $\tau(\mathcal{C})\geq \nu(\mathcal{C})$. The clutter $\mathcal{C}$ {\it packs} if $\tau(\mathcal{C})=\nu(\mathcal{C})$ \cite{Seymour77}. Observe that if a clutter is MFMC, then it packs.

\begin{DE}[\cite{Cornuejols94}]
$\mathcal{C}$ is an \emph{ideal} clutter if the primal $(P)$ has an integral optimal solution for every cost vector $w\in \mathbb{Z}^V_+$.
\end{DE}

Ideal clutters form a rich class of objects, one that contains the class of MFMC clutters, as discussed above. This containment is strict, and in fact, some of the richest examples of ideal clutters are those that are not MFMC~\cite{Guenin01,Lucchesi78}.

A clutter is {\it intersecting} if every two members intersect yet no element belongs to every member~\cite{Abdi-int}. That is, a clutter $\mathcal{C}$ is intersecting if $\tau(\mathcal{C})\geq 2$ and $\nu(\mathcal{C})=1$. In particular, an intersecting clutter does not pack, and therefore is not MFMC. Intersecting clutters, however, may be ideal. For instance, the clutter $$Q_6:=\{\{1,3,6\},\{1,4,5\},\{2,3,5\},\{2,4,6\}\},$$ whose elements are the edges and whose members are the triangles of $K_4$, is an intersecting clutter that is ideal~\cite{Seymour77}. In fact, $Q_6$ is the smallest intersecting clutter which is ideal~(\cite{Abdi-cuboids}, Proposition 1.2).

The fact that intersecting clutters may be ideal is counterintuitive, because such clutters are blatantly non-MFMC. Nonetheless, we expect our intuition to be close to the truth. In this paper, we study a sequence of stricter versions of the intersecting condition, and conjecture that this sequence of notions eventually leads to non-idealness. 

\begin{DE}
$\mathcal{C}$ is \emph{$k$-wise intersecting} if every subset of at most $k$ members have a common element, yet no element belongs to all members. 
\end{DE}

Note that for $k=2$, this notion coincides with the notion of intersecting clutters. Furthermore, for $k\geq 3$, a $k$-wise intersecting clutter is also $(k-1)$-wise intersecting. The following is our main conjecture.  

\begin{CN}\label{main-con}
There exists an integer $k\geq 4$ such that every $k$-wise intersecting clutter is non-ideal.
\end{CN}


In fact, we conjecture that the above holds for $k=4$. We prove this for an important class of clutters. A clutter is {\it binary} if the symmetric difference of any odd number of members contains a member~\cite{Lehman64}. Many rich classes of clutters are in fact binary~\cite{Cornuejols01}. For example, given a graph $G=(V,E)$ and distinct vertices $s$ and $t$, the clutter of $st$-paths over ground set $E$ is binary. The clutter $Q_6$ is also binary. As evidence for Conjecture~\ref{main-con}, our main result is that it holds for all binary clutters.  

\begin{THM}\label{main}
Every $4$-wise intersecting binary clutter is non-ideal.
\end{THM}

We also show that $4$ cannot be replaced by $3$ in Conjecture~\ref{main-con}, even for binary clutters.  

\begin{PR}\label{pete}
There exists an ideal $3$-wise intersecting binary clutter.
\end{PR}

The example from Proposition~\ref{pete} comes from the Petersen graph, and also coincides with the clutter $T_{30}$ from \cite{Schrijver03}, \S79.3e. It has $30$ elements and is the smallest such example that we are aware of.

\paragraph{Another strengthening of the intersecting condition.} For an integer $k\geq 2$, a clutter is \emph{$k$-intersecting} if every pair of members have at least $k$ elements in common, yet no element belongs to all members. One may propose an analogue of Conjecture~\ref{main-con} for $k$-intersecting clutters: \begin{quote}
(?) \emph{For a sufficiently large integer $k$, every $k$-intersecting clutter is non-ideal.} (?)
\end{quote} This conjecture, however, is false. In fact, $k$-intersecting clutters are no closer to being non-ideal than intersecting clutters. Let us elaborate.

Let $\mathcal{C}$ be a clutter over ground set $V$. To {\it duplicate an element $u$} is to introduce a new element $\bar{u}$, and replace $\mathcal{C}$ by the clutter over ground set $V\cup \{\bar{u}\}$, whose members are $\{C:C\in \mathcal{C},u\notin C\}\cup \{C\cup \{\bar{u}\} : C\in \mathcal{C},u\in C\}$. A {\it duplication of $\mathcal{C}$} is a clutter obtained from $\mathcal{C}$ by repeatedly duplicating elements. We leave it as an exercise for the reader to check that a clutter is ideal if and only if some duplication of it is ideal.

Now, if a clutter is $k$-intersecting, then the clutter obtained from it after duplicating every element once, is $2k$-intersecting. As a result, starting from an ideal intersecting clutter, such as $Q_6$, one can construct a sequence of ideal $2^k$-intersecting clutters, $k= 0,1,2,\ldots$.
 
 \subsection{Paper outline}
 
 \begin{center}
\begin{tabular}{ | m{0.25cm} | m{14.5cm} |}
\hline
 \S\ref{sec:dyadic} & As further evidence for Conjecture~\ref{main-con}, we also show that it follows for $k=5$ from an unpublished conjecture by Seymour from 1975 that was documented in \cite{Schrijver03}, \S79.3e.  \\ 
 \hline
 \S\ref{sec:chromatic} & The theory of blocking clutters allows us to give a very attractive reformulation of Conjecture~\ref{main-con} in terms of the chromatic number of a clutter. \\  
 \hline
 \S\ref{sec:cuboids} & We show that a special class of clutters, called {\it cuboids}, sit at the heart of Conjecture~\ref{main-con}. Cuboids allow us to reformulate Conjecture~\ref{main-con} yet again, but this time in terms of set systems.    \\
 \hline
 \S\ref{sec:proof} & We prove Theorem~\ref{main} and Proposition~\ref{pete}. Besides the theory of cuboids, two other key ingredients of our proof of Theorem~\ref{main} are Jaeger's {\it $8$-flow Theorem}~\cite{Jaeger79} for graphs, and Seymour's characterization of the binary matroids with the {\it sums of circuits property}~\cite{Seymour81}. \\  
 \hline
 \S\ref{sec:tackle} & We propose a line of attack for tackling Conjecture~\ref{main-con} via a deep connection to {\it projective geometries over the two-element field}, objects that give rise to $k$-wise intersecting clutters.  \\  
 \hline
 \S\ref{sec:apps} & We discuss two applications of Theorem~\ref{main}, one to uniform cycle covers in graphs, another to quarter-integral packings of value two in ideal binary clutters. \\  
 \hline
\end{tabular}
\end{center}
 


\section{Dyadic fractional packings in ideal clutters}\label{sec:dyadic}

In this section, we prove that Conjecture~\ref{main-con} for $k=5$ follows from a conjecture of Seymour. 

Let $\mathcal{C}$ be a clutter over ground set $V$. A {\it fractional packing of $\mathcal{C}$} is a vector $y\in \mathbb{R}^{\mathcal{C}}_+$ such that for all $v\in V$, $\sum\left(y_C:C\in \mathcal{C},v\in C\right)\leq 1$; its {\it value} is ${\bf 1}^\top y$. Observe that a fractional packing is nothing but a feasible solution to (D), below.
\begin{align*}
(P)\quad&\begin{array}{ll} \min \quad  &{\bf 1}^\top x\\ \text{s.t.} & \sum \left(x_u:u\in C\right)\geq 1 \quad \forall C\in \mathcal{C} \\ & x\geq {\bf 0} \end{array}\\
(D)\quad&\begin{array}{ll} \max \quad  &{\bf 1}^\top y\\ \text{s.t.} & \sum \left(y_C:u\in C\in \mathcal{C}\right)\leq 1 \quad \forall u\in V\\ & y\geq {\bf 0}.\end{array}
\end{align*} As a result, Weak LP Duality implies that every fractional packing has value at most $\tau(\mathcal{C})$. Moreover, if $\mathcal{C}$ is an ideal clutter, then Strong LP Duality guarantees the existence of a fractional packing of value $\tau(\mathcal{C})$. 
Seymour conjectures the following.

\begin{CN}[Seymour 1975, see \cite{Schrijver03}, \S79.3e]\label{seymour-quarter}
Every ideal clutter $\mathcal{C}$ has a $\frac{1}{4}$-integral packing of value $\tau(\mathcal{C})$.
\end{CN}

\noindent Here, a vector is {\it $\frac14$-integral} if every entry is $\frac14$-integral. 

Let us prove that Conjecture~\ref{seymour-quarter} implies Conjecture~\ref{main-con}. Let $I$ and $J$ be disjoint subsets of $V$. The {\it minor} $\mathcal{C}\setminus I/J$ obtained after {\it deleting $I$} and {\it contracting $J$} is the clutter over ground set $V-(I\cup J)$ whose members are the minimal sets in $\{C-J:C\in \mathcal{C},C\cap I=\emptyset\}$. If $J=\emptyset$, then $\mathcal{C}\setminus I/J=\mathcal{C}\setminus I$ is called a {\it deletion minor}. If a clutter is ideal, then so is every minor of it~\cite{Seymour77}.

\begin{PR}\label{quarter->intersecting}
If Conjecture~\ref{seymour-quarter} is true, then Conjecture~\ref{main-con} is true for $k=5$.
\end{PR}
\begin{proof}
Assume Conjecture~\ref{seymour-quarter} is true. Let $\mathcal{C}$ be an ideal clutter with $\tau(\mathcal{C})\geq 2$. We need to exhibit at most five members without a common element. By Conjecture~\ref{seymour-quarter}, $\mathcal{C}$ has a $\frac{1}{4}$-integral packing $y\in \mathbb{R}^{\mathcal{C}}_+$ of value $\tau(\mathcal{C})\geq 2$. 
Pick a minimal subset $\mathcal{C}'\subseteq \{C\in \mathcal{C}:y_C>0\}$ such that $\sum_{C\in \mathcal{C}'} y_C>1$. Clearly, $|\mathcal{C}'|\leq 5$. Moreover, the members of $\mathcal{C}'$ do not have a common element; for if $u$ was a common element to the members of $\mathcal{C}'$, then $$1<\sum_{C\in \mathcal{C}'} y_C = \sum_{u\in C\in \mathcal{C}'} y_C \leq \sum_{u\in C\in \mathcal{C}} y_C\leq 1$$ where the last inequality holds because $y$ is a fractional packing, a contradiction. As a result, $\mathcal{C}'$, and therefore $\mathcal{C}$, has a subset of at most five members without a common element, as required. 
\end{proof}

The careful reader will notice that we proved something stronger above: Conjecture~\ref{main-con} holds for $k=5$ as long as there exists a fractional packing of value at least two, where every nonzero fraction assigned is $\geq\frac{1}{4}$.

Conjecture~\ref{seymour-quarter} is notoriously difficult as it is open even for binary clutters. In fact, it remains open for the clutter of {\it postman sets} of a graph~(\cite{Cornuejols01}, Conjecture~2.15). Proposition~\ref{quarter->intersecting} is particularly interesting as it suggests a clue for proving Conjecture~\ref{seymour-quarter}, by tackling Conjecture~\ref{main-con} first. As we see in \S\ref{sec:apps}, Theorem~\ref{main} implies Conjecture~\ref{seymour-quarter} for ideal binary clutters with covering number two.


\section{The chromatic number of a clutter}\label{sec:chromatic}

In this section, we give a very attractive reformulation of Conjecture~\ref{main-con}.

Let $\mathcal{C}$ be a clutter over ground set $V$ where every member has cardinality at least two. For an integer $k\geq 2$, a {\it proper $k$-colouring of $\mathcal{C}$} is an assignment of $k$ colours to the elements $V$ such that no monochromatic member exists, that is, it is a partition of $V$ into $k$ parts none of which contains a member.

\begin{DE} The \emph{chromatic number of $\mathcal{C}$}, denoted $\chi(\mathcal{C})$, is the smallest integer $k$ such that a proper $k$-colouring exists.\end{DE}

\noindent Since every member of $\mathcal{C}$ has cardinality at least two, $\chi(\mathcal{C})$ is well-defined.

The {\it incidence matrix of $\mathcal{C}$}, denoted $M(\mathcal{C})$, is the matrix whose columns are indexed by the elements and whose rows are the incidence vectors of the members. We say that $\mathcal{C}$ is a {\it balanced} clutter if $M(\mathcal{C})$ does not have the adjacency matrix of an odd circuit as a submatrix. Balanced clutters form an important subclass of ideal clutters. If $\mathcal{C}$ is a balanced clutter where every member has cardinality at least two, then $\chi(\mathcal{C})\leq 2$ (and therefore $=2$). Balancedness, as well as the statements just made, are due to Claude Berge~\cite{Berge72}.

Given Berge's result in our context, one may wonder whether a universal upper bound exists for the chromatic number of an ideal clutter? We conjecture the following extension.

\begin{CN}\label{chromatic-con}
There exists an integer $k\geq 4$ such that every ideal clutter without a member of cardinality at most one has chromatic number at most $k$.
\end{CN}

We prove that this conjecture is, in fact, equivalent to Conjecture~\ref{main-con}! To show this, we need to introduce an important clutter notion.

Let $\mathcal{C}$ be a clutter over ground set $V$. The family of minimal covers of $\mathcal{C}$ forms another clutter over the ground set $V$, called the {\it blocker of $\mathcal{C}$} and denoted $b(\mathcal{C})$. It is well-known that $b(b(\mathcal{C}))=\mathcal{C}$~\cite{Isbell58,Edmonds70}. 

\begin{PR}\label{kwise<->colouring}
Let $\mathcal{C}$ be a clutter over ground set $V$ where every member has cardinality at least two. Let $k\geq 2$ be an integer. Then $\mathcal{C}$ has chromatic number at most $k$ if, and only if, $b(\mathcal{C})$ is \emph{not} $k$-wise intersecting.
\end{PR}
\begin{proof}
$(\Rightarrow)$ 
Let $A_1,\ldots,A_k$ be a proper $k$-colouring of $\mathcal{C}$, that is, $A_1,\ldots,A_k$ form a partition of $V$ such that no $A_i,i\in [k]:=\{1,\ldots,k\}$ contains a member of $\mathcal{C}$. Then for each $i\in [k]$, $V-A_i$ is a cover of $\mathcal{C}$, so there exists $B_i\in b(\mathcal{C})$ disjoint from $A_i$. Since $A_1\cup \cdots\cup A_k=V$, it follows that $B_1\cap \cdots\cap B_k=\emptyset$, certifying that $b(\mathcal{C})$ is not $k$-wise intersecting.

$(\Leftarrow)$
Assume that $b(\mathcal{C})$ is not $k$-wise intersecting. Since every member of $\mathcal{C}$ has cardinality at least two, $\tau(b(\mathcal{C}))\geq 2$, so there exist $B_1,\ldots,B_k\in b(\mathcal{C})$ such that $B_1\cap \cdots\cap B_k=\emptyset$. Let $A_1:=V-B_1$ and for $i=2,\ldots,k$ let
$$A_i:=(B_1\cap \cdots\cap B_{i-1}) - B_i.$$ By definition, $A_1,\ldots,A_k$ are pairwise disjoint. As $B_1\cap \cdots \cap B_k=\emptyset$, it follows that $A_1,\ldots,A_k$ partition $V$. As no $A_i$ is a cover of $b(\mathcal{C})$, it follows that $A_1,\ldots,A_k$ is a proper $k$-colouring of $b(b(\mathcal{C}))=\mathcal{C}$, so $\mathcal{C}$ has chromatic number at most $k$.
\end{proof}

A clutter is ideal if, and only if, its blocker is ideal~\cite{Fulkerson70,Lehman79,Cornuejols01}. We are now ready to prove the following.

\begin{PR}
Conjecture~\ref{main-con} for $k$ is equivalent to Conjecture~\ref{chromatic-con} for the same~$k$.
\end{PR}
\begin{proof}
Assume that Conjecture~\ref{main-con} holds for $k$, that is, every $k$-wise intersecting clutter is non-ideal. Let $\mathcal{C}$ be an ideal clutter where every member has cardinality at least two. Then $b(\mathcal{C})$ is an ideal clutter, so it is not $k$-wise intersecting by our hypothesis, so $\chi(\mathcal{C})\leq k$ by Proposition~\ref{kwise<->colouring}. Thus, Conjecture~\ref{chromatic-con} holds for the same $k$. The other direction is similar.
\end{proof}

A result that must be noted in the context of Conjecture~\ref{chromatic-con} is the following: For every ideal clutter $\mathcal{C}$ where every member and also every cover has cardinality at least two, either $\chi(\mathcal{C})\leq 2$ or $\chi(b(\mathcal{C}))\leq 2$ (or both). This result was proved recently, and the proof relied on {\it Gauge Duality} in {\it Quadratic Programming}~\cite{Abdi-ISBC}.

As a final note, it is well-known that a clutter is binary if, and only if, the blocker is binary~\cite{Lehman64}. As a consequence, Theorem~\ref{main} and Proposition~\ref{pete} (which have not yet been proved) are equivalent, respectively, to the following statements.

\begin{THM}
Every ideal binary clutter without a member of cardinality at most one has chromatic number at most $4$.
\end{THM}

\begin{PR}
There exists an ideal binary clutter whose chromatic number is equal to $4$.
\end{PR}


\section{Cuboids}\label{sec:cuboids}

In this section, we reformulate Conjecture~\ref{main-con} in terms of cuboids. The concept of cuboids is key for the proof of Theorem~\ref{main}.

Let $n \geq 1$ be an integer, and let $S\subseteq \{0,1\}^n$. The {\it cuboid of $S$}, denoted $\cuboid{S}$, is the clutter over ground set $[2n]$ whose members have incidence vectors $(p_1,1-p_1,\ldots,p_n,1-p_n)$ over all $(p_1, \dots, p_n) \in S$.  We say that a clutter is a {\it cuboid} if, after a relabeling of its ground set, is equal to $\cuboid{S}$, for some $S$.

 Observe that for each $C\in \cuboid{S}$, $|C\cap \{2i-1,2i\}|= 1$ for all $i\in [n]$. In particular, every member of $\cuboid{S}$ has size $n$ (hence $\cuboid{S}$ is a clutter) and $\tau(\cuboid{S}) \leq 2$. Cuboids were introduced in~\cite{Abdi-mnp} and further studied in~\cite{Abdi-cuboids}. 
 
 We now describe what it means for $\cuboid{S}$ to be $k$-wise intersecting. We say that the points in $S$ {\it agree on a coordinate} if $S\subseteq \{x : x_i=a\}$ for some coordinate $i\in [n]$ and some $a\in \{0,1\}$.

\begin{RE}\label{kwise-kwise}
Let $S\subseteq \{0,1\}^n$. Then $\cuboid{S}$ is a $k$-wise intersecting clutter if, and only if, the points in $S$ do not agree on a coordinate yet every subset of at most $k$ points do.
\end{RE}
\begin{proof}
$(\Rightarrow)$ Assume that $\cuboid{S}$ is $k$-wise intersecting. 

If the points in $S$ agree on a coordinate, say $S\subseteq \{x : x_i=a\}$, then the members of $\cuboid{S}$ would have an element in common, either $2i-1$ (if $a=1$) or $2i$ (if $a=0$), which is not possible because the members of $\cuboid{S}$ do not have a common element. Thus, the points in $S$ do not agree on a coordinate.

If a subset $S'$ of at most $k$ points in $S$ do agree on a coordinate, say $S'\subseteq \{x : x_i=a\}$, then the members of $\cuboid{S}'\subseteq \cuboid{S}$ would have an element in common, leading to at most $k$ members in $\cuboid{S}$ that have a common element, which again is not possible because $\cuboid{S}$ is $k$-wise intersecting. Thus, every subset of $S$ of at most $k$ points agree on a coordinate.

$(\Leftarrow)$ is similar and left to the reader.
\end{proof}

Next, we describe what it means for $\cuboid{S}$ to be ideal. Let $\conv(S)$ denote the convex hull of $S$. An inequality of the form $\sum_{i\in I} x_i + \sum_{j\in J} (1-x_j)\geq 1$, for some disjoint $I,J\subseteq [n]$, is called a {\it generalized set covering inequality}~\cite{Cornuejols01}. The set $S$ is {\it cube-ideal} if every facet of $\conv(S)$ is defined by $x_i\geq 0$, $x_i\leq 1$, or a generalized set covering inequality~\cite{Abdi-cuboids}.

\begin{THM}[\cite{Abdi-cuboids}]\label{ideal-cuboids}
Let $S\subseteq \{0,1\}^n$. Then $\cuboid{S}$ is an ideal clutter if, and only if, $S$ is a cube-ideal set.
\end{THM}

As a result, (the contrapositive of) Conjecture~\ref{main-con} for cuboids reduces to the following conjecture.

\begin{CN}\label{main-con-3}
There exists an integer $k\geq 4$ such that for every cube-ideal set, either all the points agree on a coordinate, or there is a subset of at most $k$ points that do not agree on a coordinate.
\end{CN}

Surprisingly, we show that Conjecture~\ref{main-con-3} is equivalent to Conjecture~\ref{main-con}! We need the notions of a {\it tangled} clutter and the {\it core} of an ideal tangled clutter.

\subsection{Tangled clutters}

Take an integer $k\geq 2$. If every $k$ members of a clutter have a common element, then so do every $k$ members of a deletion minor. Furthermore, for every element $v\in V$, $\tau(\mathcal{C})\geq \tau(\mathcal{C}\setminus v)\geq \tau(\mathcal{C})-1$, where $\tau(\mathcal{C}\setminus v) = \tau(\mathcal{C})-1$ if and only if $v$ belongs to some minimum cover of $\mathcal{C}$. Motivated by these observations, we make the following definition.

\begin{DE}
$\mathcal{C}$ is a \emph{tangled} clutter if $\tau(C)=2$ and every element belongs to a minimum cover.
\end{DE}

Observe that every cuboid with covering number two is a tangled clutter. Tangled clutters are relevant due to the following remark.
 
\begin{RE}\label{kwise-deletion}
Let $\mathcal{C}$ be a $k$-wise intersecting clutter. Let $\mathcal{C}'$ be a deletion minor of $\mathcal{C}$ that is minimal subject to $\tau(\mathcal{C}')\geq 2$. Then $\mathcal{C}'$ is a tangled $k$-wise intersecting clutter.
\end{RE}
\begin{proof}
It is clear that $\mathcal{C}'$ is a tangled clutter. As $\mathcal{C}'$ is a deletion minor of $\mathcal{C}$, every member of $\mathcal{C}'$ is also a member of $\mathcal{C}$, so every subset of at most $k$ members of $\mathcal{C}'$ have a common element, implying in turn that $\mathcal{C}'$ is a $k$-wise intersecting clutter.
\end{proof}

Thus, it suffices to prove Conjecture~\ref{main-con} for tangled clutters. In the rest of this section, we further reduce the conjecture to cuboids.

\subsection{Every ideal tangled clutter has an ideal core.}

Let $\mathcal{C}$ be an ideal tangled clutter over ground set $V$. Consider the following dual pair of linear programs:
\begin{align*}
(P)\quad &\begin{array}{ll} \min \quad  &{\bf 1}^\top x\\ \text{s.t.} & \sum \left(x_u:u\in C\right)\geq 1 \quad \forall C\in \mathcal{C} \\ & x\geq {\bf 0} \end{array}\qquad
(D)\quad&\begin{array}{ll} \max \quad  &{\bf 1}^\top y\\ \text{s.t.} & \sum \left(y_C:u\in C\in \mathcal{C}\right)\leq 1 \quad \forall u\in V\\ & y\geq {\bf 0}\end{array}
\end{align*} As $\mathcal{C}$ has covering number two, (P) has optimal value two, and every minimum cover of $\mathcal{C}$ yields an optimum to (P). By Strong LP Duality, the optimal value of (D) is also two; let $y^\star$ be a fractional packing of value two. By Complementary Slackness, whenever $y^\star_C>0$, then $|C\cap \{u,v\}|=1$ for every minimum cover $\{u,v\}$ of $\mathcal{C}$. This observation motivates the following definition.


\begin{DE}
Let $\mathcal{C}$ be an ideal tangled clutter over ground set $V$. The \emph{core of $\mathcal{C}$} is the clutter $$\core{\mathcal{C}}:=\{C\in \mathcal{C} : |C\cap \{u,v\}| = 1 \text{ for every minimum cover \{u,v\}}\}.
$$ 
\end{DE}

In this subsection, we prove that the core of every ideal tangled clutter is another ideal tangled clutter, one that arises from a cuboid. We need a few ingredients. 

Let $\mathcal{C}$ be a clutter over ground set $V$. Let $$
Q(\mathcal{C}) := \left\{x\in \mathbb{R}^V_+:\sum_{v\in C}x_v\geq 1\quad C\in \mathcal{C}\right\}.
$$ Observe that the $0-1$ points in $Q(\mathcal{C})$ are precisely the incidence vectors of the covers of $\mathcal{C}$, while the integral vertices of $Q(\mathcal{C})$ are precisely the incidence vectors of the minimal covers of $\mathcal{C}$. Basic polyhedral theory tells us that $\mathcal{C}$ is an ideal clutter if, and only if, $Q(\mathcal{C})$ is another integral polyhedron~(see \cite{Conforti14}, Theorem 4.1).

Recall that to duplicate an element $u$ of $\mathcal{C}$ is to introduce a new element $\bar{u}$, and replace $\mathcal{C}$ by the clutter over ground set $V\cup \{\bar{u}\}$, whose members are $\{C:C\in \mathcal{C},u\notin C\}\cup \{C\cup \{\bar{u}\} : C\in \mathcal{C},u\in C\}$. Recall further that a duplication of $\mathcal{C}$ is a clutter obtained from $\mathcal{C}$ by repeatedly duplicating elements. It can be readily checked that a clutter is $k$-wise intersecting if and only if some duplication of it is $k$-wise intersecting.

\begin{PR}\label{core-prop}
Let $\mathcal{C}$ be an ideal tangled clutter over ground set $V$. Let $G=(V,E)$ be the graph whose edges correspond to the minimum covers of $\mathcal{C}$. Then the following statements hold: \begin{enumerate}
\item $\core{\mathcal{C}} = \{C\in \mathcal{C} : y_C>0 \text{ for some fractional packing $y$ of value two}\}$,
\item $G$ is a bipartite graph, and
\item if $\{U,U'\}$ is the bipartition of a connected component of $G$, then the elements in $U$ (resp. $U'$) are duplicates in $\core{\mathcal{C}}$, and $|\{u,u'\}\cap C|=1$ for all $u\in U,u'\in U'$ and $C\in \core{\mathcal{C}}$.
\end{enumerate}
\end{PR}
\begin{proof}
{\bf (1)} $(\supseteq)$ follows from Complementary Slackness while $(\subseteq)$ follows from Strict Complementarity.
{\bf (2)} It follows from (1) that $\core{\mathcal{C}}\neq \emptyset$. It can now be readily checked that $G$ is a bipartite graph, as for each $C\in \core{\mathcal{C}}$, $\{C,V-C\}$ is a bipartition of $G$.
{\bf (3)} If $\{u,v\},\{u,w\}$ are minimum covers of $\mathcal{C}$, then $v,w$ are duplicates in $\core{\mathcal{C}}$. This observation proves (3).
\end{proof}

We need the following lemma.

\begin{LE}[\cite{Abdi-mnp}, Lemma 3.1]\label{idealcore-LE}
Let $\mathcal{C}$ be a clutter whose ground set can be partitioned into parts $\{u_1,v_1\},\ldots,$ $\{u_r,v_r\}$ such that $|\{u_i,v_i\}\cap C|=1$ for each $i\in [r]$ and $C\in \mathcal{C}$. Then $\mathcal{C}$ is ideal if, and only if, $\conv\big\{\chi_C:C\in \mathcal{C} \big\}=Q\big(b(\mathcal{C})\big) \cap \big\{x: x_{u_i}+x_{v_i}=1 ~\forall i\in [r]\big\}$. 
\end{LE}

Here, $\chi_C$ denotes the incidence vector of $C$. 

We are now ready to prove the main result of this subsection.

\begin{THM}\label{ideal-core}
Let $\mathcal{C}$ be an ideal tangled clutter. Then $\core{\mathcal{C}}$ is a duplication of a cuboid. Moreover, $\core{\mathcal{C}}$ is an ideal tangled clutter.
\end{THM}
\begin{proof}
Denote by $V$ the ground set of $\mathcal{C}$. Let $G=(V,E)$ be the graph whose edges correspond to the minimum covers of $\mathcal{C}$. By Proposition~\ref{core-prop}~(2), $G$ is a bipartite graph. Let $r$ be the number of connected components of $G$, and for each $i\in [r]$, let $\{U_i,V_i\}$ be the bipartition of the $i\textsuperscript{th}$ connected component. By Proposition~\ref{core-prop}~(3), for each $i\in [r]$, the elements in $U_i$ are duplicates in $\core{\mathcal{C}}$, the elements in $V_i$ are duplicates in $\core{\mathcal{C}}$, and $|\{u,v\}\cap C|=1$ for all $u\in U_i,v\in V_i$ and $C\in \core{\mathcal{C}}$. 
That is, each $C \in \core{\mathcal{C}}$ is determined by $r$ binary choices; in each connected component of $G$, $C$ must contain exactly one of the two parts of the bipartition. This allows a more concise representation of the core. For each $C
  \in \core{\mathcal{C}}$, define $p_{C} \in \{0, 1\}^{r}$ such that
  $$(p_{C})_{i} = \begin{cases}
    0 & \text{if } C \cap (U_{i} \cup V_{i}) = V_{i}\\
    1 & \text{if } C \cap (U_{i} \cup V_{i}) = U_{i}
  \end{cases}$$ Let $S := \{p_{C} : C \in \core{\mathcal{C}}\}\subseteq \{0,1\}^r$. Then $\core{\mathcal{C}}$ is a duplication of $\cuboid{S}$.

\begin{CLM} 
$\core{\mathcal{C}}$ is a tangled clutter.
\end{CLM}
\begin{cproof}
As a subset of $\mathcal{C}$, $\core{\mathcal{C}}$ has covering number at most two, and every element of it appears in a cover of cardinality two. Thus, to prove the claim, it suffices to show that $\core{\mathcal{C}}$ has covering number at least two. Let $y$ be a fractional packing of $\mathcal{C}$ of value two. Then $\supp(y)\subseteq \core{\mathcal{C}}$ by Proposition~\ref{core-prop}~(1), so $y$ is also a fractional packing of $\core{\mathcal{C}}$. Subsequently, $\core{\mathcal{C}}$ has covering number at least two, as required.
\end{cproof}

It remains to prove that $\core{\mathcal{C}}$ is an ideal clutter; we use Lemma~\ref{idealcore-LE} to prove this. We know that
\begin{equation}\tag{$\star$}
\{\chi_C:C\in \core{\mathcal{C}}\} = \{\chi_C:C\in \mathcal{C}\}\cap \big\{x: x_u+x_v = 1,~\{u,v\}\in E\big\}.
\end{equation}

\begin{CLM} 
$\conv\{\chi_C:C\in \core{\mathcal{C}}\}= Q\big(b(\core{\mathcal{C}})\big) \cap \big\{x: x_{u}+x_{v}=1, ~\{u,v\}\in E\big\}$.
\end{CLM}
\begin{cproof}
 $(\subseteq)$ follows immediately from $(\star)$. $(\supseteq)$ Pick a point $x^\star$ in the set on the right-hand side. As $Q\big(b(\core{\mathcal{C}})\big)\subseteq Q\big(b(\mathcal{C})\big)$, we have $x^\star\in Q\big(b(\mathcal{C})\big)$. Since $\mathcal{C}$ is ideal, so is $b(\mathcal{C})$, implying that for some $\lambda\in \mathbb{R}_+^{\mathcal{C}}$ with $\sum_{C\in \mathcal{C}} \lambda_C=1$, we have that $$x^\star \geq \sum_{C\in \mathcal{C}} \lambda_C \chi_{C}.$$ Since for all $\{u,v\}\in E$, we have that $x^\star_{u}+x^\star_{v}=1$ and $\{u,v\}\in b(\mathcal{C})$, equality must hold above and by $(\star)$, if $\lambda_C>0$ then $C\in \core{\mathcal{C}}$. Hence, $x^\star\in \conv\{\chi_C:C\in \core{\mathcal{C}}\}$, as required.
\end{cproof}

For each $i\in [r]$, pick $u_i\in U_i$ and $v_i\in V_i$, and let $\mathcal{C}'$ be the clutter over ground set $\{u_1,v_1,\ldots,u_r,v_r\}$ obtained from $\core{\mathcal{C}}$ after contracting $V-\{u_1,v_1,\ldots,u_r,v_r\}$. Notice that $|\{u_i,v_i\}\cap C|=1$ for all $i\in [r]$ and $C\in \mathcal{C}'$. (Observe that $\mathcal{C}'$ is nothing but $\cuboid{S}$.) 

\begin{CLM} 
$\conv\{\chi_C:C\in \mathcal{C}' \}=Q\big(b(\mathcal{C}')\big) \cap \big\{z: z_{u_i}+z_{v_i}=1 ~ i\in [r]\big\}$.
\end{CLM}
\begin{cproof}
We use Claim~2 to prove this equality. Observe that $\conv\{\chi_C:C\in \mathcal{C}' \}$ is the projection of $\conv\{\chi_C:C\in \core{\mathcal{C}}\}$ onto the coordinates $\{u_i,v_i:i\in [r]\}$. Thus, to finish the proof, it suffices to show that $Q\big(b(\mathcal{C}')\big) \cap \big\{z: z_{u_i}+z_{v_i}=1,~ i\in [r]\big\}$ is the projection of $Q\big(b(\core{\mathcal{C}})\big) \cap \big\{x: x_{u}+x_{v}=1, ~\{u,v\}\in E\big\}$ onto the same coordinates. We leave this as an easy exercise for the reader.
\end{cproof}

It therefore follows from Lemma~\ref{idealcore-LE} that $\mathcal{C}'$ is an ideal clutter. As $\core{\mathcal{C}}$ is a duplication of $\mathcal{C}'$, $\core{\mathcal{C}}$ is an ideal clutter, too, thereby completing the proof.
\end{proof}

\subsection{Conjectures~\ref{main-con-3} and~\ref{main-con} are equivalent.}

\begin{THM}\label{con1-con3}
Conjecture~\ref{main-con-3} for $k$ is equivalent to Conjecture~\ref{main-con} for the same~$k$.
\end{THM}
\begin{proof}
We already showed $(\Leftarrow)$.
It remains to prove $(\Rightarrow)$. Suppose Conjecture~\ref{main-con} is false for some $k\geq 4$. That is, there is an ideal $k$-wise intersecting clutter $\mathcal{C}$. Let $\mathcal{C}'$ be a deletion minor of $\mathcal{C}$ that is minimal subject to $\tau(\mathcal{C}')\geq 2$. By Remark~\ref{kwise-deletion}, $\mathcal{C}'$ is an ideal tangled $k$-wise intersecting clutter. 
By Theorem~\ref{ideal-core}, $\core{\mathcal{C}'}$ is an ideal tangled clutter that is a duplication of some cuboid, say $\cuboid{S}$.
As every $k$ members of $\mathcal{C}'$ have a common element, so do every $k$ members of $\core{\mathcal{C}'}$, so the latter is $k$-wise intersecting. As a result, $\cuboid{S}$ is an ideal $k$-wise intersecting clutter, so by Remark~\ref{kwise-kwise} and Theorem~\ref{ideal-cuboids}, $S$ is a cube-ideal set whose points do not agree on a coordinate yet every subset of $\leq k$ points do. Therefore, $S$ refutes Conjecture~\ref{main-con-3} for $k$, as required.
\end{proof}


\section{Graphs, binary matroids, and binary clutters}\label{sec:proof}

In this section, we prove Theorem~\ref{main} and Proposition~\ref{pete}. More specifically, in \S\ref{sec:graphs}, we explain how every graph leads to a cube-ideal set and how the $8$-Flow Theorem proves Conjecture~\ref{main-con-3} for such cube-ideal sets; we also give an example of an ideal $3$-wise intersecting binary clutter, thereby proving Proposition~\ref{pete}. After a primer on binary matroids in \S\ref{sec:primer}, we introduce the class of binary matroids with the sums of circuits property~\ref{sec:binary-matroids}. Using the machinery developed in earlier subsections, we finally prove Theorem~\ref{main} in \S\ref{sec:main-proof}.

A key notion throughout this section is the following. Let $S\subseteq \{0,1\}^n$. For $x,y\in \{0,1\}^n$, $x\tr y$ denotes the coordinate-wise sum of $x,y$ modulo $2$. We say that $S$ is a {\it vector space over $GF(2)$}, or simply a {\it binary space}, if $a\tr b\in S$ for all $a,b\in S$. Notice that a nonempty binary space necessarily contains ${\bf 0}$. 

\begin{RE}\label{binaryclutter<->binaryspace}
Let $S\subseteq \{0,1\}^n$ contain ${\bf 0}$. Then $\cuboid{S}$ is a binary clutter if, and only if, $S$ is a binary space.
\end{RE}

Recall that a clutter is binary if the symmetric difference of any odd number of members contains a member. A binary clutter can be characterized as the {\it port} of a binary matroid, and also as the clutter of the {\it odd circuits of a signed binary matroid} (see \cite{Guenin02}). The above remark shows yet another way to obtain a binary clutter from a binary matroid.

\subsection{The $8$-Flow Theorem}\label{sec:graphs}

Let $G=(V,E)$ be a graph where loops and parallel edges are allowed, where every loop is treated as an edge not incident to any vertex. A {\it cycle} is a subset $C\subseteq E$ such that every vertex is incident with an even number of edges in $C$. A {\it bridge of $G$} is an edge $e$ that does not belong to any cycle. The {\it cycle space of $G$} is the set $$\cycle{G}:=\left\{\chi_C : C\subseteq E \text{ is a cycle}\right\}\subseteq \{0,1\}^E.$$ As $\emptyset$ is a cycle, and the symmetric difference of any two cycles is also a cycle, it follows that $\cycle{G}$ is a binary space.  We require the following two results on cycle spaces of graphs.  

\begin{RE}\label{graph-cyclespace}
Let $G=(V,E)$ be a graph. Then the following statements hold: \begin{enumerate}
\item The points in $\cycle{G}$ agree on a coordinate if, and only if, $G$ has a bridge.
\item For all $k\in \mathbb{N}$, $\cycle{G}$ has a subset of at most $k+1$ points that do not agree on a coordinate if, and only if, $G$ has at most $k$ cycles the union of which is $E$. 
\end{enumerate}
\end{RE}
\begin{proof}
{\bf (1)} $(\Rightarrow)$ Since ${\bf 0}\in S$, we must have that $\cycle{G}\subseteq \{x:x_e=0\}$ for some edge $e\in E$, so $e$ is not contained in any cycle of $G$, implying in turn that $e$ is a bridge. $(\Leftarrow)$ is left as an exercise. {\bf (2)} $(\Leftarrow)$ Pick $k$ cycles $C_1,\ldots,C_k$ whose union is $E$. Then the $k+1$ points ${\bf 0},\chi_{C_1},\ldots,\chi_{C_k}$, which all belong to $\cycle{G}$, do not agree on a coordinate. $(\Rightarrow)$ Pick $k+1$ points $p_1,p_2,\ldots,p_k,p_{k+1}$ in $\cycle{G}$ that do not agree on a coordinate. Then $p_1\tr p_{k+1},p_2\tr p_{k+1},\ldots,p_k\tr p_{k+1},{\bf 0}$, all of which belong to $\cycle{G}$, do not agree on a coordinate either. Pick cycles $C_1,\ldots,C_k$ such that $\chi_{C_i}=p_i\tr p_{k+1}$ for $i\in [k]$. It can be readily checked that $C_1\cup \ldots\cup C_k=E$, thereby finishing the proof.
\end{proof}

\begin{THM}[\cite{Seymour79,Barahona86}, see \cite{Abdi-cuboids}]\label{graph-cubeideal}
The cycle space of every graph is a cube-ideal set.
\end{THM}

We need the following version of the celebrated {\it $8$-Flow Theorem} of Jaeger~\cite{Jaeger79}. (The second equivalent statement follows from Remark~\ref{graph-cyclespace}.)

\begin{THM}[\cite{Jaeger79}]\label{jaeger}
Let $G=(V,E)$ be a graph. Then either $G$ has a bridge, or there are most $3$ cycles the union of which is $E$. That is, given $\cycle{G}\subseteq \{0,1\}^E$, either all the points agree on a coordinate, or there is a subset of at most $4$ points that do not agree on a coordinate.
\end{THM}

One may wonder whether the $3,4$ in Theorem~\ref{jaeger} may be replaced by $2,3$? The answer is no, due to the Petersen graph (see Figure~\ref{fig:pete}):

\begin{RE}[see \cite{Tutte66}]\label{Tutte}
The edge set of the Petersen graph is not the union of $2$ cycles.  
\end{RE}
\begin{proof}
Suppose for a contradiction that the edge set, $E$, of the Petersen graph is the union of two cycles $C_1,C_2$. Then $J_1:=E-C_1,J_2:=E-C_2$ are disjoint {\it postman sets}, i.e. the odd-degree vertices of $J_i$ coincide with the odd-degree vertices of the graph. Since $J_1\cap J_2=\emptyset$, and the graph is cubic, $J_1$ and $J_2$ must be perfect matchings. Clearly, $E-(J_1\cup J_2)$ is another perfect matching, so the Petersen graph is $3$-edge-colourable, a contradiction.
\end{proof}

\begin{figure}[h]
\begin{subfigure}[t]{0.21\textwidth}
\includegraphics[width=\linewidth]{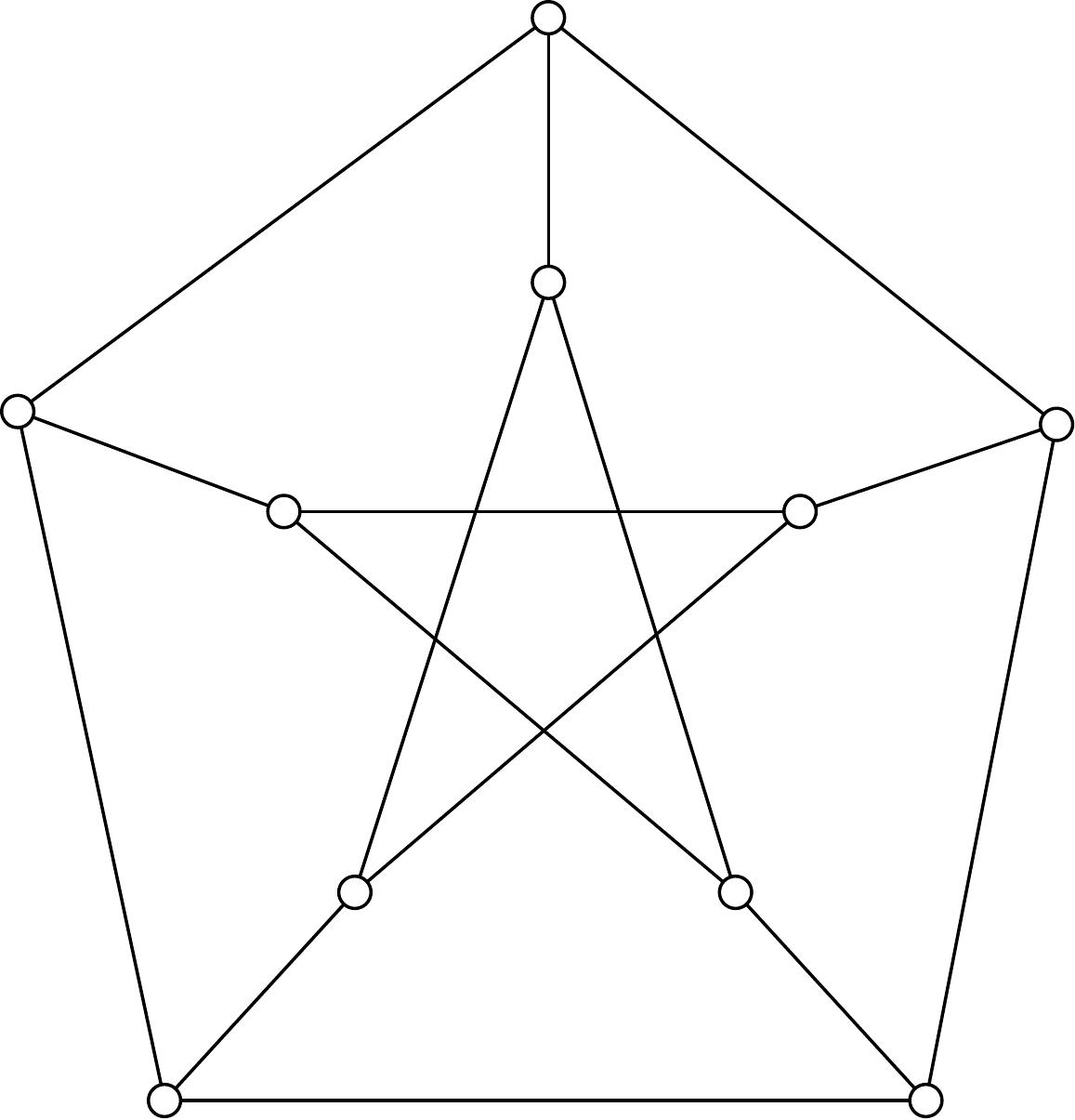}
\caption{The Petersen graph} \label{fig:pete}
\end{subfigure}
\hspace*{\fill}
\begin{subfigure}[t]{0.31\textwidth}
\includegraphics[width=\linewidth]{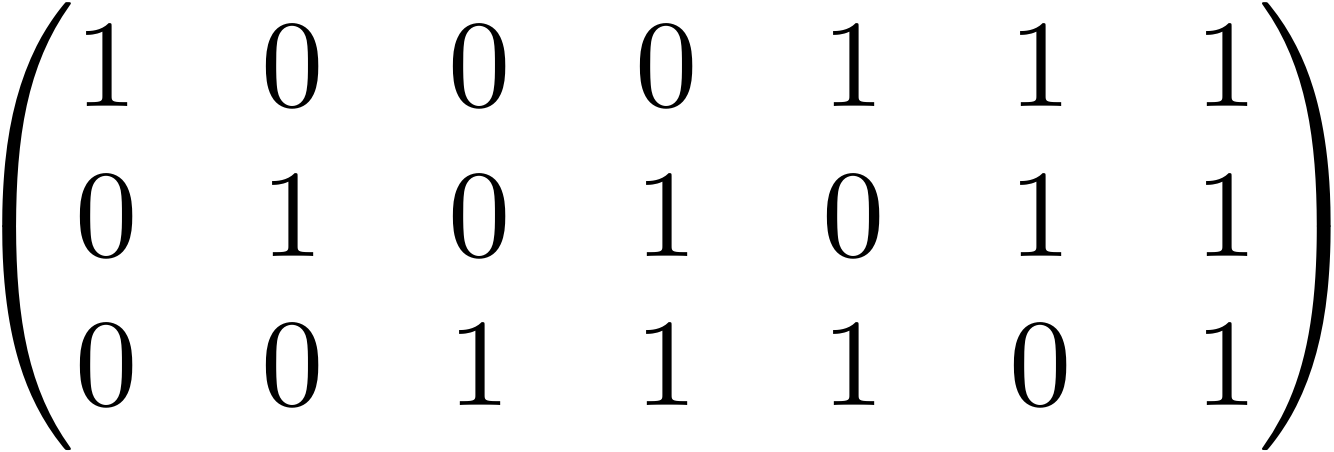}
\caption{Representation of the Fano matroid} \label{fig:fano}
\end{subfigure}
\hspace*{\fill}
\begin{subfigure}[t]{0.21\textwidth}
\includegraphics[width=\linewidth]{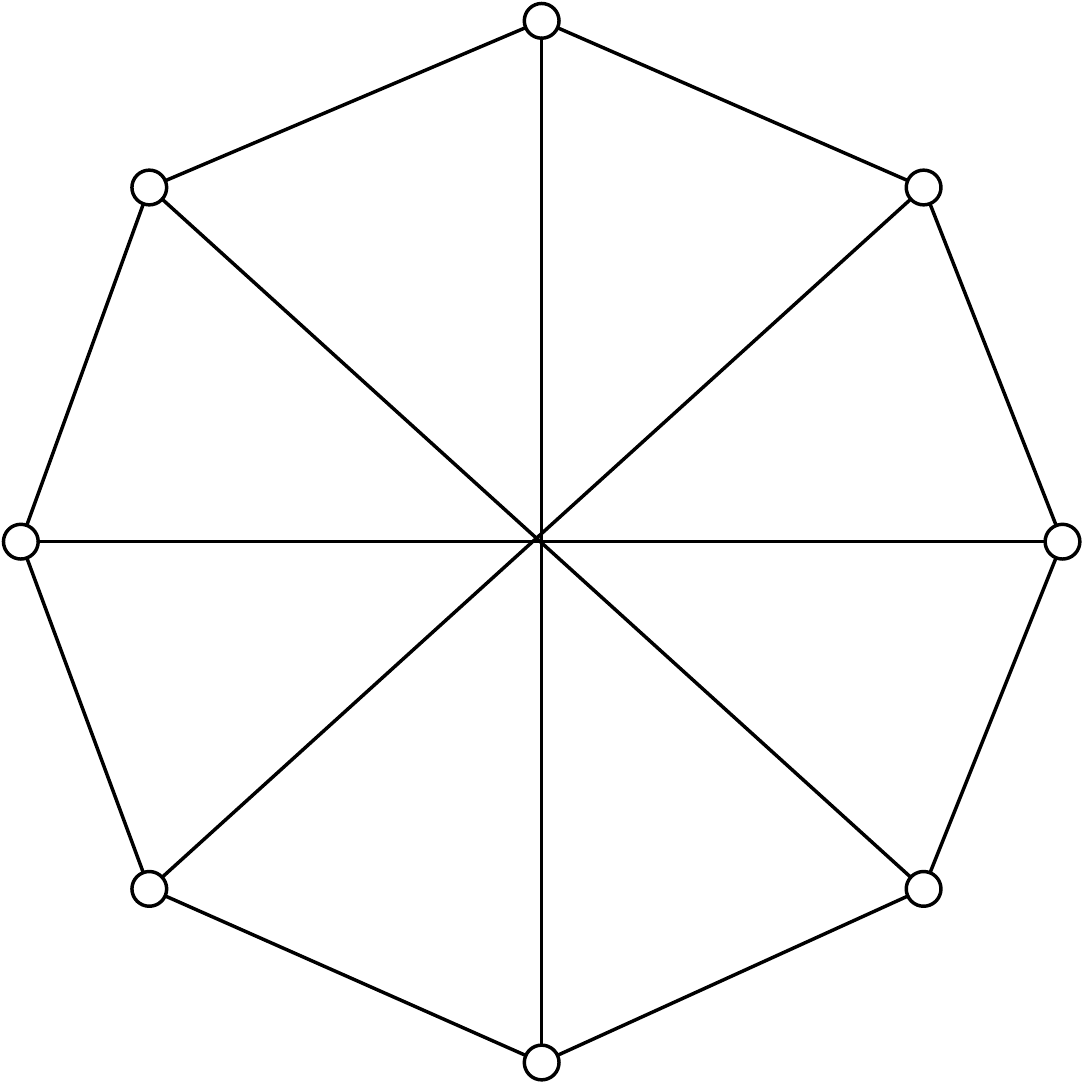}
\caption{The Wagner graph}
\label{fig:wagner}
\end{subfigure}
\caption{}
\end{figure}

As a consequence, an ideal $3$-wise intersecting clutter does exist:

\begin{proof}[Proof of Proposition~\ref{pete}]
Let $S$ be the cycle space of the Petersen graph, and let $\mathcal{C}:=\cuboid{S}$. By Remark~\ref{binaryclutter<->binaryspace}, $\mathcal{C}$ is a binary clutter. By Remark~\ref{Tutte}, the Petersen is a bridgeless graph that does not have $2$ cycles the union of which is the edge set, so by Remark~\ref{graph-cyclespace}, the points in $S$ do not agree on a coordinate, but every subset of $2+1=3$ points do. Moreover, $S$ is a cube-ideal set by Theorem~\ref{graph-cubeideal}. Therefore, by Remark~\ref{kwise-kwise} and Theorem~\ref{ideal-cuboids}, $\mathcal{C}$ is an ideal $3$-wise intersecting binary clutter, as required.
\end{proof}

The cuboid of the cycle space of the Petersen graph has already shown up in the literature, and is denoted $T_{30}$ by Schrijver~\cite{Schrijver03}, \S79.3e. Consider the graph obtained from Petersen by subdividing every edge once, and let $T$ be any vertex subset of even cardinality containing all the new vertices. Then the clutter of \emph{minimal $T$-joins} of this graph is precisely $T_{30}$. (This is left as a nice exercise for the reader.) This construction is due to Seymour (\cite{Seymour79b}, page 440).

\subsection{A primer on binary matroids}\label{sec:primer}

We follow Oxley~\cite{Oxley11}. Let $E$ be a finite set, $S\subseteq \{0,1\}^E$ a binary space, and $S^\perp$ the orthogonal complement of $S$, that is, $S^\perp=\{y\in \{0,1\}^E:y^\top x \equiv 0 \pmod{2} ~ \forall x\in S\}$. Notice that $S^\perp$ is another binary space, and that $(S^\perp)^\perp =S$. Therefore, there exists a $0-1$ matrix $A$ whose columns are labeled by $E$ such that $S = \left\{x\in \{0,1\}^E:Ax\equiv {\bf 0} \pmod{2}\right\}$, and $S^\perp$ is the row space of $A$ generated over $GF(2)$.

Let $\mathcal{S}:=\{C\subseteq E:\chi_C\in S\}$. The pair $M:=(E,\mathcal{S})$ is a {\it binary matroid}, and the matrix $A$ is a {\it representation of $M$}. We call $E$ the {\it ground set of $M$}, and denote it by $EM$. The sets in $\mathcal{S}$ are the {\it cycles of~$M$}, and $\mathcal{S}$ is the {\it cycle space of~$M$}, denoted by $\cycle{M}$. The minimal nonempty sets in $\mathcal{S}$ are the {\it circuits of~$M$}, and the circuits of cardinality one are {\it loops}.

Let $\mathcal{S}^\perp:=\{D\subseteq E:\chi_D\in S^\perp\}$. The binary matroid $M^\star:=(E,\mathcal{S}^\perp)$ is the {\it dual of $M$}. Notice that $(M^\star)^\star=M$. The sets in $\mathcal{S}^\perp$ are the {\it cocycles of~$M$}, and $\mathcal{S}^\perp$ is the {\it cocycle space of $M$}, denoted by $\cocycle{M}$. The minimal nonempty sets in $\mathcal{S}^\perp$ are the {\it cocircuits of~$M$}, and the cocircuits of cardinality one are {\it coloops of $M$}.

\begin{RE}\label{matroid-cyclespace}
Let $M$ be a binary matroid. Then the following statements hold: \begin{enumerate}
\item The points in $\cycle{M}$ agree on a coordinate if, and only if, $M$ has a coloop.
\item For all $k\in \mathbb N$, $\cycle{M}$ has a subset of at most $k+1$ points that do not agree on a coordinate if, and only if, $M$ has at most $k$ cycles the union of which is $EM$.
\end{enumerate}
\end{RE}

Let $G=(V,E)$ be a graph. The binary matroid whose cycle space is $\cycle{G}$ is a {\it graphic matroid}, and is denoted $M(G)$. Notice the one-to-one correspondence between the cycles of $M(G)$ and the cycles of $G$, between the loops of $M(G)$ and the loops of $G$, between the cocycles of $M(G)$ and the cuts of $G$, and between the coloops of $M(G)$ and the bridges of $G$.  Therefore, Remark~\ref{matroid-cyclespace} is an extension of Remark~\ref{graph-cyclespace} to all binary matroids. (The proofs of the two remarks are almost identical.)

Finally, let $M$ be a binary matroid. A pair of distinct elements $e,f\in EM$ are {\it parallel} if $\{e,f\}$ is a circuit of $M$. Given distinct elements $e,f,g \in EM$, if $e,f$ are parallel and $f,g$ are parallel, then so are $e,g$. A {\it parallel class} of $M$ is a maximal subset of $EM$ of pairwise parallel elements. The {\it simplification of $M$}, denoted $\si{M}$, is the binary matroid obtained from $M$ after deleting all loops, and for every parallel class, keeping one representative and deleting all the other elements. By construction, $\si{M}$ is a {\it simple} binary matroid, i.e. it has no circuit of cardinality at most two.

\subsection{The sums of circuits property}\label{sec:binary-matroids}

A binary matroid $M$ over ground set $E$ has the {\it sums of circuits property} if for each $w\in \mathbb{R}^E_+$ satisfying $$w(D-\{f\})\geq w_f \quad \text{ for every cocycle $D$ and $f\in D$,}$$ there exists an assignment $y_C\geq 0$ to every circuit $C$ such that $$w = \sum \left(y_C\cdot \chi_C: C \text{ is a circuit}\right).$$ This notion was first introduced by Seymour where he proved that graphic matroids have the sums of circuits property~\cite{Seymour79}. This matroid theoretic notion is relevant as it coincides with the notion of cube-idealness for binary spaces.

\begin{THM}[\cite{Barahona86,Abdi-cuboids}]\label{soc-cubeideal}
A binary matroid has the sums of circuits property if, and only if, the corresponding cycle space is a cube-ideal set.
\end{THM}

Seymour proved a decomposition theorem for binary matroids with the sums of circuits property~\cite{Seymour81}. It turns out they can all be produced from graphic matroids and two other matroids, which we now describe.  The {\it Fano matroid $F_7$} is the binary matroid represented by the matrix in Figure~\ref{fig:fano}. The second matroid is $M(V_8)^\star$, where $V_8$ is the graph in Figure~\ref{fig:wagner}. Seymour showed that $F_7$ and $M(V_8)^\star$ both have the sums of circuits property~\cite{Seymour81}. 

To generate all binary matroids with the sums of circuits property, we require three composition rules. Let $M_1,M_2$ be binary matroids over ground sets $E_1,E_2$, respectively. We denote by $M_1\tr M_2$ the binary matroid over ground set $E_1\tr E_2$ whose cycles are all subsets of $E_1\tr E_2$ of the form $C_1\tr C_2$, where $C_i$ is a cycle of $M_i$ for $i \in [2]$. Then $M_1\tr M_2$ is a {\it $1$-sum} if $E_1\cap E_2= \emptyset$; $M_1\tr M_2$ is a {\it $2$-sum} if $E_1\cap E_2=\{e\}$, where $e$ is neither a loop nor a coloop of $M_1$ or $M_2$; and $M_1\tr M_2$ is a {\it $Y$-sum} if $E_1\cap E_2$ is a cocircuit of cardinality $3$ in both $M_1$ and $M_2$ and contains no circuit in $M_1$ or $M_2$.

\begin{THM}[\cite{Seymour81}, (6.4), (6.7), (6.10) and (16.4)]\label{seymour-soc}
Let $M$ be a binary matroid with the sums of circuits property. Then $M$ is obtained recursively by means of $1$-sums, $2$-sums and $Y$-sums starting from copies of $F_7,M(V_8)^\star$ and graphic matroids.
\end{THM}

We are now ready to extend Theorem~\ref{jaeger} to all binary matroids with the sums of circuits property, as follows. (The second equivalent statement follows from Remark~\ref{matroid-cyclespace} and Theorem~\ref{soc-cubeideal}.)

\begin{THM}\label{soc-3cyclecover}
Every binary matroid without a coloop and with the sums of circuits property has at most $3$ cycles the union of which is the ground set. That is, given a cube-ideal binary space, either all the points agree on a coordinate, or there is a subset of at most $4$ points that do not agree on a coordinate.
\end{THM}
\begin{proof}
A {\it $3$-cycle cover} of a binary matroid is $3$ (not necessarily distinct) cycles whose union is the ground set. 

\setcounter{claim_nb}{0} 

\begin{CLM} 
Both $F_7$ and $M(V_8)^\star$ have $3$-cycle covers.
\end{CLM}
\begin{cproof}
Given the matrix representation of $F_7$ in Figure~\ref{fig:pete}, label the columns $1,\ldots,7$ from left to right. Then  $\emptyset,\{1,2,3,7\},\{4,5,6\}$ is a $3$-cycle cover of $F_7$. Next, label the vertices of $V_8$ so that the outer $8$-cycle is labelled $1,\dots,8$.  Then $M(V_8)^\star$ has a $3$-cycle cover given by the following cuts of $V_8$: $\delta(\{1,6,7,8\}),$ $\delta(\{1,7\}),\delta(\{2,4\})$, where $\delta(X)$ is the set of edges with exactly one end in $X$.
\end{cproof}


\begin{CLM} 
Let $M,M_1,M_2$ be binary matroids such that $M=M_1\tr M_2$ and $M_i,i \in [2]$ has a $3$-cycle cover. Then the following statements hold:\begin{enumerate}
\item If $M$ is a $1$-sum of $M_1,M_2$, then $M$ has a $3$-cycle cover.
\item If $M$ is a $2$-sum of $M_1,M_2$, then $M$ has a $3$-cycle cover.
\item If $M$ is a $Y$-sum of $M_1,M_2$, then $M$ has a $3$-cycle cover.
\end{enumerate}
\end{CLM}
\begin{cproof}
For $i\in [2]$, let $E_i$ be the ground set of $M_i$ and $C_1^i,C_2^i,C_3^i$ be a $3$-cycle cover of $M_i$.
Clearly, {\bf (1)} holds.
For {\bf (2)}, let $E_1\cap E_2=\{e\}$. We may assume $e \in C_1^i$ for all $i \in [2]$. By replacing $C_2^i$ by $C_1^i\tr C_2^i$ if necessary, we may assume $e \notin C_2^i$ for all $i \in [2]$. Similarly, we may assume $e \notin C_3^i$ for $i \in [2]$. But now $\{C_j^1\tr C_j^2:j\in [3]\}$ is a $3$-cycle cover of $M$. 
For {\bf (3)}, suppose $E_1\cap E_2=\{e,f,g\}$. Since $\{e,f,g\}$ is a cocircuit of both $M_1,M_2$, and since cocircuits and circuits of a binary matroid have an even number of elements in common, $|C_j^i\cap \{e,f,g\}|\in \{0,2\}$ for all $i,j$. Therefore, after possibly relabeling $e,f,g$ simultaneously in $M_1$ and $M_2$, and after possibly relabeling $C_1^i,C_2^i,C_3^i$ for all $i$, we may assume that \begin{itemize}
\item $C_1^i\cap \{e,f,g\}=\{e,f\}$ for all $i \in [2]$, and
\item $C_2^i\cap \{e,f,g\}=\{e,g\}$ or $\{f,g\}$ for all $i\in [2]$.
\end{itemize}
For  $i \in [2]$, after possibly replacing $C_2^i$ with $C_2^i\tr C_1^i$, we may assume $C_2^i\cap \{e,f,g\}=\{e,g\}$.
For  $i \in [2]$, after possibly replacing $C_3^i$ with $C_3^i\tr C_1^i,C_3^i\tr C_2^i$ or $C_3^i\tr C_1^i\tr C_2^i$, we may assume $C_3^i\cap \{e,f,g\}=\emptyset$. But now $\{C_j^1\tr C_j^2:j\in [3]\}$ is a $3$-cycle cover of $M$, as required. 
\end{cproof}

We leave the proof of the following claim as an easy exercise for the reader.  

\begin{CLM} 
Let $M,M_1,M_2$ be binary matroids such that $M=M_1\tr M_2$, where $\tr$ is either a $1$-, $2$- or $Y$-sum. If $M$ has no coloop, then neither do $M_1,M_2$. 
\end{CLM}

The proof is completed by combining the above claims with Theorems~\ref{jaeger} and~\ref{seymour-soc}.  
\end{proof}

\subsection{Proof of Theorem~\ref{main}}\label{sec:main-proof}

It is well-known that a clutter is binary if, and only if, every member and every minimal cover have an odd number of elements in common~\cite{Lehman64}. We use this characterization to prove the following.

\begin{PR}\label{binaryclutter->cuboid}
Let $\mathcal{C}$ be a binary tangled clutter. Then $\mathcal{C}$ is a duplication of a cuboid.
\end{PR}
\begin{proof}
If $\{e,f\}$ is a minimum cover, then for each $C\in \mathcal{C}$, $|C\cap \{e,f\}|$ must be odd and therefore $1$, since $\mathcal{C}$ is a binary clutter. As a result, if $\{e,f\},\{e,g\}$ are both minimum covers, then $f,g$ must be duplicates. Moreover, if every element is contained in exactly one minimum cover, then $\mathcal{C}$ must be a cuboid. These two observations, along with the fact that $\mathcal{C}$ is a tangled clutter, imply that $\mathcal{C}$ is a duplication of a cuboid.
\end{proof}

Every minor of a binary clutter is also a binary clutter~\cite{Seymour76}. We are now equipped to prove the main result of this paper, stating that every $4$-wise intersecting binary clutter is non-ideal.

\begin{proof}[Proof of Theorem~\ref{main}]
We prove the contrapositive statement. Let $\mathcal{C}$ be an ideal binary clutter such that $\tau(\mathcal{C}) \geq 2$. We need to exhibit $\leq 4$ members without a common element. Let $\mathcal{C}'$ be a deletion minor of $\mathcal{C}$ that is minimal subject to $\tau(\mathcal{C}')\geq 2$. It suffices to exhibit $\leq 4$ members of $\mathcal{C}'$ without a common element. Notice that $\mathcal{C}'$ is ideal, and as a minor of a binary clutter, it is also binary. Moreover, by our minimality assumption, $\mathcal{C}'$ is a tangled clutter. Thus, by Proposition~\ref{binaryclutter->cuboid}, $\mathcal{C}'$ is a duplication of a cuboid, say $\cuboid{S}$ where we may choose $S$ so that ${\bf 0}\in S$. It suffices to exhibit $\leq 4$ members of $\cuboid{S}$ without a common element.

Note that $\cuboid{S}$ is an ideal binary cuboid with $\tau(\cuboid{S}) \geq 2$. So, by Theorem~\ref{ideal-cuboids} and Remark~\ref{binaryclutter<->binaryspace}, $S$ is a cube-ideal binary space whose points do not agree on a coordinate.
By Theorem~\ref{soc-3cyclecover}, $S$ has $\leq 4$ points that do not agree on a coordinate, thereby yielding $\leq 4$ members of $\cuboid{S}$ without a common element, as required.
\end{proof}


\section{Projective geometries over the two-element field}\label{sec:tackle}

In this section, we give an important example of $k$-wise intersecting clutters, coming from projective geometries over the two-element field. We also propose a strengthening of Conjecture~\ref{main-con}. Using Theorem~\ref{main}, we prove our stronger conjecture for the class of binary clutters.

Conjecture~\ref{main-con} predicts that for some $k\geq 4$, every ideal clutter with covering number at least two has $k$ members without a common element. By moving to a deletion minor, if necessary, we may assume that our ideal clutter is tangled. Roughly speaking, our stronger conjecture predicts that the tangled deletion minor must actually have $2^{k-1}$ members that come from a projective geometry, and of these members, $k$ many will not have a common element.

Take an integer $\ell\geq 1$. Let $A$ be the $\ell\times (2^\ell-1)$ matrix whose columns are all the nonzero vectors in $\{0,1\}^\ell$. The binary matroid represented by $A$ is called a {\it projective geometry over $GF(2)$}, and is denoted $PG(\ell-1,2)$. See Figure~\ref{fig:PG-reps} for representations of the first three projective geometries. Note that $PG(1,2)$ is the graphic matroid of a triangle, while $PG(2,2)$ is the Fano matroid $F_7$. \begin{figure}
\centering
$$
\begin{pmatrix}
1
\end{pmatrix}\qquad
\begin{pmatrix}
1&0&1\\
0&1&1
\end{pmatrix}\qquad
\begin{pmatrix}
1&0&0&0&1&1&1\\
0&1&0&1&0&1&1\\
0&0&1&1&1&0&1
\end{pmatrix}
$$
\caption{Representations of $PG(0,2),PG(1,2),PG(2,2)$, from left to right.}
\label{fig:PG-reps}
\end{figure}

Projective geometries are relevant as they give rise to an important class of $\ell$-wise intersecting clutters, as we see below. Given integers $n,m\geq 1$ and some points $a_1,\ldots,a_m\in \{0,1\}^n$, denote by $\langle a_1,\ldots,a_m\rangle$ the vector space over $GF(2)$ generated by the points $a_1,\ldots,a_m$.

\begin{PR}\label{PG}
Take an integer $\ell\geq 1$. Then the following statements hold: \begin{enumerate}
\item $PG(\ell-1,2)$ has rank $\ell$.
\item $PG(\ell-1,2)$ has exactly $2^\ell$ cocycles.
\item $PG(\ell-1,2)$ has a unique representation, up to permuting rows and columns.
\item Every nonempty cocycle of $PG(\ell-1,2)$ has cardinality $2^{\ell-1}$.
\end{enumerate} Moreover, for $S:=\cocycle{PG(\ell-1,2)}$, the following statements hold: \begin{enumerate}
\item[5.] 
$\cuboid{S}$ has $\ell+1$ members without a common element.
\item[6.] 
$\cuboid{S}$ is an $\ell$-wise intersecting clutter.
\end{enumerate}
\end{PR}
\begin{proof}
Let $A$ be the representation of $PG(\ell-1,2)$, i.e. the $\ell\times (2^\ell-1)$ matrix whose columns are all the nonzero vectors in $\{0,1\}^\ell$. {\bf (1)} and {\bf (2)} follow from the facts that the rows of $A$ are linearly dependent over $GF(2)$, and that $\cocycle{PG(\ell-1,2)}$ is the row space of $A$ generated over $GF(2)$. 
{\bf (3)} follows from {\bf (1)} and the fact that $PG(\ell-1,2)$ has $2^\ell-1$ elements, has no loop, and has no parallel elements. 
{\bf (4)} 
Every nonempty cocycle can be viewed as the first row of $A$, the unique representation of $PG(\ell-1,2)$, and since the rows of $A$ have precisely $2^{\ell-1}$ ones, (4) follows. 
{\bf (5)}
Note that the $\ell$ points in $S$ corresponding to the rows of $A$ agree on precisely one coordinate, which is set to $1$. These $\ell$ points, together with the zero point ${\bf 0}$, yield $\ell+1$ points that do not agree on a coordinate.
{\bf (6)} 
Let $a_1,\ldots,a_{\ell}$ be some points in $S$. Let $r$ be the rank of $\langle a_1,\ldots,a_{\ell}\rangle$. Clearly, $0\leq r\leq \ell$. If $r=\ell$, then by (3), the $\ell$ points may be viewed as the $\ell$ rows of $A$, so they agree on precisely one coordinate, which is set to $1$. Otherwise, $r<\ell$. We may assume that $\{a_1,\ldots,a_r\}$ form a basis for $\langle a_1,\ldots,a_{\ell}\rangle$. Then, by (3), $\{a_1,\ldots,a_r\}$ may be viewed as $r$ distinct rows of $A$, implying that they agree on $2^{\ell-r}-1$ coordinates set to $0$, implying in turn that all the points in $\{a_1,\ldots,a_{\ell}\}$ agree on $2^{\ell-r}-1\geq 1$ coordinates set to $0$, as required.
\end{proof}

Observe that $\cocycle{PG(0,2)} = \{0,1\}$ and $\cocycle{PG(1,2)}=\{000,011,$ $101,110\}$. In particular, \begin{align*}
\cuboid{\cocycle{PG(0,2)}} &= \{\{1\},\{2\}\}\\
\cuboid{\cocycle{PG(1,2)}} &= \{\{2,4,6\},\{2,3,5\},\{1,4,5\},\{1,3,6\}\}=Q_6.
\end{align*}

\subsection{Embedding projective geometries}

A clutter $\mathcal{C}$ {\it embeds $PG(k-2,2)$} if some subset of $\mathcal{C}$ is a duplication of the cuboid of $\cocycle{PG(k-2,2)}$. 

\begin{RE}\label{embed-small-PGs}
Let $\mathcal{C}$ be a clutter over ground set $V$. The following statements hold: \begin{enumerate}
\item $\mathcal{C}$ embeds $PG(0,2)$ if, and only if, $\mathcal{C}$ has two members that partition $V$,
\item $\mathcal{C}$ embeds $PG(1,2)$ if, and only if, $\mathcal{C}$ has four members of the form $I_2\cup I_4\cup I_6, I_2\cup I_3\cup I_5, I_1\cup I_4\cup I_5,I_1\cup I_3\cup I_6$, where $I_1,\ldots,I_6$ are nonempty sets that partition $V$
\end{enumerate}
\end{RE}

Cornu\'{e}jols, Guenin and Margot refer to the equivalent condition in Remark~\ref{embed-small-PGs}~(2) as the {\it $Q_6$ property}~\cite{Cornuejols00}. In that paper, they prove that a subclass of ideal tangled clutters, namely, ideal {\it minimally non-packing} clutters with covering number two, enjoys the $Q_6$ property. We propose the following conjecture extending their result to all ideal tangled clutters.

\begin{CN}\label{PG-2}
There exists an integer $\ell\geq 3$ such that every ideal tangled clutter embeds one of $PG(0,2),\ldots,$ $PG(\ell-1,2)$.
\end{CN}

Let us prove that this conjecture is a strengthening of our main conjecture.

\begin{PR} \label{strengthening}
If Conjecture~\ref{PG-2} holds for $\ell$, then Conjecture~\ref{main-con} holds for $k=\ell+1$.
\end{PR}
\begin{proof}
Assume Conjecture~\ref{PG-2} holds for $\ell$. Let $\mathcal{C}$ be an ideal clutter with $\tau(C) \geq 2$. We need to exhibit $\leq \ell+1$ members without a common element. Let $\mathcal{C}'$ be a deletion minor of $\mathcal{C}$ that is minimal subject to $\tau(\mathcal{C}')\geq 2$. Then $\mathcal{C}'$ is an ideal tangled clutter. Thus, $\mathcal{C}'$ embeds $PG(n-1,2)$ for some $n\in [\ell]$. That is, a duplication of $\cuboid{PG(n-1,2)}$ is a subset of $\mathcal{C}'$. By Proposition~\ref{PG}~(5), $\cuboid{PG(n-1,2)}$ has $n+1$ members without a common element, so the duplication, and therefore $\mathcal{C}'$, must have $n+1$ members without a common element. Thus, $\mathcal{C}$ has $n+1\leq \ell+1$ members without a common element, as required.
\end{proof}

\subsection{Conjecture~\ref{PG-2} is true for binary clutters.}

\begin{PR}\label{PG-triangle}
Let $M$ be a simple binary matroid where every two elements appear together in a triangle. Then $M$ is a projective geometry.
\end{PR}
\begin{proof}
We may assume that $EM=[r]$ for some integer $r\geq 1$. Let $A$ be a $0-1$ matrix with column labels $[r]$ and whose rows are linearly independent over $GF(2)$, where $$\cycle{M}=\{x : Ax\equiv {\bf 0} \pmod{2}\}.$$ After elementary row operations over $GF(2)$ and reordering the columns, if necessary, we may assume that $A=(I ~\vert~ A')$ where $I$ is the identity matrix of appropriate dimension. As $M$ is simple, it follows that every column of $A'$ has at least two $1$s, and no two columns of it are equal. Since every two elements of $M$ appear together in a triangle, it follows that if $a,b$ are distinct columns of $A$, then $a+b \pmod{2}$ is another column. It can now be readily checked that $A'$ consists of all $0-1$ vectors with at least two $1$s, implying in turn that $M$ is a projective geometry, as required.
\end{proof}

Given an integer $n\geq 2$ and a set $S\subseteq \{0,1\}^n$, two distinct coordinates $i,j\in [n]$ are {\it duplicates in $S$} if $S\subseteq \{x : x_i=x_j\}$ or $S\subseteq \{x : x_i + x_j=1\}$. Observe that $S$ has duplicated coordinates if, and only if, $\cuboid{S}$ has duplicated elements. We say that $S$ is a {\it duplication of $S'$} if $S'$ is obtained from $S$ after projecting away some duplicated coordinates.

\begin{PR}\label{binaryspace->PG}
Take an integer $n\geq 1$ and let $S\subseteq \{0,1\}^n$ be a binary space whose points do not agree on a coordinate. Then there exists a subset $S'\subseteq S$, whose points do not agree on a coordinate, that is a duplication of the cocycle space of a projective geometry. As a consequence, if $S$ has $GF(2)$-rank $r$, then $\cuboid{S}$ embeds one of $PG(0,2),\ldots,PG(r-1,2)$.
\end{PR}
\begin{proof}
Let $A$ be a $0-1$ matrix whose rows $a_1,\ldots,a_m$ are linearly independent over $GF(2)$ such that $S=\{x\in \{0,1\}^n : Ax\equiv {\bf 0} \pmod{2}\}$. Clearly $\langle a_1,\ldots,a_m\rangle$ is the orthogonal complement of $S$ over $GF(2)$. As the points in $S$ do not agree on a coordinate, $\langle a_1,\ldots,a_m\rangle$ does not contain the unit vectors $e_1,\ldots,e_n$ (see Remark~\ref{matroid-cyclespace}~(1)). Extend $\{a_1,\ldots,a_m\}$ to a set $\{a_1,\ldots,a_m,\ldots,a_k\}$ of $0-1$ vectors such that \begin{enumerate}[(i)]
\item $a_1,\dots,a_k$ are linearly independent over $GF(2)$,
\item $\langle a_1,\ldots,a_k\rangle$ does not contain $e_1,\ldots,e_n$,
\item $\{a_1,\ldots,a_k\}$ is maximal subject to (i)-(ii).
\end{enumerate} 
Let $M$ be the binary matroid over ground set $[n]$ whose cycle space is $\langle a_1,\ldots,a_k\rangle$. It follows from (ii) that $M$ has no loop. 

\setcounter{claim_nb}{0} 

\begin{CLM} 
For every pair of distinct elements $f,g$, there is a cycle of $M$ of cardinality at most three containing $f,g$.
\end{CLM}
\begin{cproof}
Let $a:=\chi_{\{f,g\}}\in \{0,1\}^n$. If $a\in \langle a_1,\ldots,a_k\rangle$, then $\{f,g\}$ is a cycle of $M$. Otherwise, $a_1,\ldots,a_k,a$ are linearly independent over $GF(2)$, so by (iii), $\langle a_1,\ldots,a_k,a\rangle$ must contain one of $e_1,\ldots,e_n$. Neither $f$ nor $g$ is a loop of $M$, so $\langle a_1,\ldots,a_k\rangle$ must contain a vector $b$ with three $1$s such that $b_f=b_g=1$. That is, $M$ has a triangle containing $f,g$, as required. 
\end{cproof}

Let $M':=\si{M}$, the simplification of $M$. Then $\cocycle{M}$ is a duplication of $\cocycle{M'}$. Claim 1 implies that every two elements of $M'$ appear together in a triangle, so $M'$ is a projective geometry by Proposition~\ref{PG-triangle}. Thus, $\cocycle{M}$ is a duplication of the cocycle space of a projective geometry. Notice however that $\cocycle{M}=\{x\in \{0,1\}^n:Bx\equiv {\bf 0} \pmod{2}\}$ where $B$ is the matrix whose rows are $a_1,\ldots,a_k$. Subsequently, $S':=\cocycle{M}\subseteq S$. As $M$ has no loop, the points in $S'$ do not agree on a coordinate, so $S'$ is the desired set. 
\end{proof}

As an application of Proposition~\ref{binaryspace->PG}, Conjecture~\ref{PG-2} can be rephrased as, \emph{There exists an integer $\ell\geq 3$ such that every ideal tangled clutter ``embeds a binary matroid" of rank at most $\ell$}. As a second application of Proposition~\ref{binaryspace->PG}, we prove Conjecture~\ref{PG-2} for $\ell=3$ for the class of binary clutters.

\begin{THM}\label{main-strong}
Every ideal binary tangled clutter embeds $PG(0,2)$, $PG(1,2)$, or $PG(2,2)$.
\end{THM}
\begin{proof}
Let $\mathcal{C}$ be a binary tangled clutter. By Proposition~\ref{binaryclutter->cuboid}, $\mathcal{C}$ is a duplication of a cuboid, say $\cuboid{S}$ for some $S$ containing ${\bf 0}$. It suffices to show that $\cuboid{S}$ embeds one of the three projective geometries. Note that $\cuboid{S}$ is an ideal binary cuboid with $\tau(\cuboid{S}) \geq 2$. By Theorem~\ref{ideal-cuboids} and Remark~\ref{binaryclutter<->binaryspace}, $S$ is a cube-ideal binary space whose points do not agree on a coordinate.  By Theorem~\ref{soc-3cyclecover}, $S$ has a subset of at most $3$ points that do not agree on a coordinate. Let $S'$ be the binary space generated by these points. Note that $S'\subseteq S$, the points in $S'$ do not agree on a coordinate, and $S'$ has $GF(2)$-rank at most $3$. By Proposition~\ref{binaryspace->PG}, $\cuboid{S'}$, and therefore $\cuboid{S}$, embeds one of $PG(0,2),PG(1,2),PG(2,2)$, as desired.
\end{proof}


\section{Two applications}\label{sec:apps}

In this final section, we discuss two applications of Theorem~\ref{main-strong}. The first application is a result in Graph Theory, that every bridgeless graph has a so-called {\it $7$-cycle $4$-cover}.

\begin{THM}[\cite{Bermond83}, Proposition 6]
Every bridgeless graph has $7$ cycles such that every edge is used in exactly $4$ of the cycles.
\end{THM}
\begin{proof}
Let $G$ be a bridgeless graph. Applying Theorem~\ref{main-strong} to $\cuboid{\cycle{G}}$, we see that $G$ has $8$ cycles where every edge is used in exactly $4$ of the cycles (some of the cycles may be repeated). Since one of the $8$ cycles may be assumed to be $\emptyset$, $G$ has $7$ cycles such that each edge is in exactly $4$ of the cycles.
\end{proof}

For the next application, we prove that Conjecture~\ref{seymour-quarter} holds for ideal binary clutters with covering number two. We need the following ingredient.

\begin{PR}[\cite{Abdi-PGs}]\label{PG->integral}
For every $k \in \mathbb{Z}_{\geq 0}$, $\cuboid{\cocycle{PG(k,2)}}$ has a $\frac{1}{2^{k}}$-integral packing of value two.
\end{PR}

We are now ready for the second application.

\begin{THM}\label{main-strong-2}
Every ideal binary clutter $\mathcal{C}$ with $\tau(\mathcal{C})\geq 2$ has a $\frac{1}{4}$-integral packing of value two.
\end{THM}
\begin{proof}
Let $\mathcal{C}'$ be a deletion minor of $\mathcal{C}$ that is minimal subject to $\tau(\mathcal{C}')\geq 2$. Then $\mathcal{C}'$ is an ideal binary tangled clutter, so by Theorem~\ref{main-strong}, $\mathcal{C}'$ embeds one of $PG(0,2),PG(1,2),PG(2,2)$. By
Proposition~\ref{PG->integral}, $\mathcal{C}'$, and therefore $\mathcal{C}$, has a $\frac{1}{4}$-integral packing of value $2$, as required.
\end{proof}


\section*{Acknowledgements}

This work was supported by NSERC PDF grant 516584-2018, ONR grant 00014-18-12129, the Australian Research Council, ERC grant FOREFRONT (grant agreement no. 615640), and the Institute for Basic Science (IBS-R029-C1). We would like to thank the referees of our manuscript and its extended abstract~\cite{Abdi-2020}; their valuable feedback has improved our presentation.

{\small \bibliographystyle{abbrv}\bibliography{references}}

\end{document}